\documentclass[12pt]{amsart}
\usepackage{amssymb,amsmath,amsthm, graphicx, fullpage, verbatim}

\newtheorem{theorem}{Theorem}
\newtheorem{lemma}[theorem]{Lemma}
\newtheorem{corollary}[theorem]{Corollary}

\numberwithin{equation}{section}
\numberwithin{Theorem}{section}

\newcommand{\R}{\mathbb{R}}

\newcommand{\Z}{\mathbb{Z}}

\newcommand{\N}{\mathbb{N}}
\newcommand{\supp}{\operatorname{supp}}

\newcommand{\p}{\partial}

\newcommand{\bs}{\backslash}
\newcommand{\ds}{\displaystyle}

\begin{document}

\title{Fourier transform of anisotropic Hardy spaces}
\author{Marcin Bownik}
\author{Li-An Daniel Wang}
\address{Department of Mathematics, University of Oregon, Eugene,
OR 97403--1222, U.S.A.}
\email{mbownik@uoregon.edu, lwang3@uoregon.edu}


\date{\today}

\begin{abstract}
We show that if $f$ is in anisotropic Hardy space $H_A^p$, $0 < p \leq 1$, with respect to a dilation matrix $A$, then its Fourier transform $\hat{f}$ satisfies pointwise estimate
\[
|\hat f(\xi)| \le C ||f||_{H^p_A} \rho_*(\xi)^{\frac{1}{p}-1}.\]
Here, $\rho_*$ is a quasi-norm associated with the transposed matrix $A^*$. This leads to necessary conditions for functions $m$ to be multipliers on $H_A^p$, as well as further pointwise characterizations on $\hat{f}$ and a generalization of the Hardy-Littlewood inequality on the integrability of $\hat{f}$. This last result is strengthened through the use of rearrangement functions.
\end{abstract}

\maketitle

\section{Introduction}

In the real-variable theory of Hardy spaces $H^p$ of Fefferman and Stein \cite{CFeSt-1972}, a well-known problem is the characterization of $\hat{f}$ for $f \in H^p$. Coifman \cite{Co-1974b} characterized all such $\hat{f}$ on $\R$ using entire functions of exponential type. In higher dimensions necessary conditions have been studied by a number of authors \cite{Col-1982, GCK-2001, TaWe-1980}. In particular, Taibleson and Weiss \cite{TaWe-1980} showed that for $p \in (0, 1]$, the Fourier transform of $f \in H^p(\R^n)$ is continuous and satisfies the following estimate:
    \begin{align}
    \label{Iso-F(f)}
    |\hat{f}(\xi)| \leq C \| f \|_{H^p} |\xi|^{n(\frac{1}{p} - 1)}.
    \end{align}
This leads to the following consequences; see \cite[III.7]{GCF-WN}, \cite{GCK-2001} for more details. At the origin, the estimate \eqref{Iso-F(f)} forces $f \in H^p \cap L^1$ to have vanishing moments, as seen by the degree of 0 of $\hat{f}$ at the origin, illustrating the necessity of the vanishing moments of the atoms. Away from the origin, the polynomial growth is sharp, as given by an extension of the Hardy-Littlewood inequality for $f\in H^p$, $0<p\le 1$,
    \begin{align}\label{int2}
    \int_{\R^n} |\xi|^{n(p - 2)} |\hat{f}(\xi)|^p d\xi \leq C \| f \|_{H^p}^p.
    \end{align}
The estimate \eqref{Iso-F(f)} also sheds light on multiplier operators of $H^p$. When paired with the molecular characterization of $H^p$, it shows that the multiplier operator $T_m : H^p \rightarrow H^p$ is bounded provided the multiplier $m$ satisfies the (integral) H\"{o}rmander condition. On the other hand, if $T_m$ is any bounded multiplier operator on $H^p$, then $m$ is necessarily continuous and bounded on $\R^n \bs \{ 0 \}$.

The main purpose of this paper is to extend \eqref{Iso-F(f)} from the isotropic (classical) setting to anisotropic Hardy spaces $H_A^p$ associated with a dilation matrix $A$. In this new setting, the continuous dilation $\varphi_t (x) = t^{-n} \varphi(x/t)$ for $t > 0$ is replaced by the discrete dilation $\varphi_k (x) = |\det A|^k \varphi(A^k x)$ for $k \in \Z$ and the Euclidean norm $|\cdot|$ is generalized by a quasinorm $\rho : \R^n \rightarrow [0, \infty)$ associated with $A$. When $A = 2I_n$, the anisotropic setting coincides with the classical theory since $\rho$ can be chosen as $\rho(x)=|x|^n$. We denote by $\rho_{\ast}$ the quasi-norm associated with the transposed matrix $A^*$. Our main result takes the following form.

\begin{theorem}\label{Thm1} Let $p \in (0, 1]$. If $f \in H_A^p (\R^n)$, then $\hat{f}$ is a continuous function and satisfies
    \begin{align}
    \label{Aniso-F(f)}
    |\hat{f}(\xi)| \leq C \| f \|_{H_A^p} \ \rho_{\ast} (\xi)^{\frac{1}{p} - 1}
    \end{align}
with $C = C(A, p)$.
\end{theorem}
Theorem \ref{Thm1} leads to similar consequences as in the isotropic setting. At the origin, we obtain a sharper order for the convergence of $\hat{f}(\xi)$ as $\xi \rightarrow 0$. This is given by Corollary \ref{Cor-Local}, and shows the necessity of vanishing moments for anisotropic atoms in $ H_A^p$. We then obtain necessary conditions for a function $m$ to be a multiplier on $H_A^p$, given by Corollary \ref{Cor-Mult}. Lastly, we show in Corollary \ref{Cor-Global} that the function $|\hat{f}(\xi)|^p \rho_{\ast} (\xi)^{p - 2}$ is integrable, which is a generalization of Hardy-Littlewood's inequality \eqref{int2}. In Theorem \ref{Lem-Rearr}, we further improve this estimate using rearrangement functions as in the work of Garc\'ia-Cuerva and Kolyada \cite{GCK-2001}, though we use a slightly different argument.

The anisotropic structure considered here was motivated by wavelet theory, and is certainly not the first generalization of the underlying $\R^n$ structure.  Calder\'{o}n and Torchinsky \cite{CaTo1-1977, CaTo2-1977} studied the parabolic setting of using dilations of continuous groups $\{ A_t \}_{t > 0}$ on $\R^n$. Folland and Stein \cite{FoSt-1971} replaced the underlying $\R^n$ with homogeneous groups, and Coifman and Weiss initiated the study of Hardy spaces on spaces of homogeneous type in their seminal work \cite{CW-1977}. However, the extension of \eqref{Iso-F(f)} was not considered in the parabolic setting, and the Fourier transform takes a more abstract form on homogeneous groups. Moreover, the Fourier transform is not even considered on spaces of homogeneous type, as these spaces might not have an underlying group structure.

In the next section, we briefly give the background on anisotropic Hardy spaces. In Section 3, we prove Theorem \ref{Thm1}. The consequences of this theorem are in Section 4.

\section{Anisotropic Setting} We now introduce the anisotropic structure and the associated Hardy spaces. For more details see Bownik \cite{Bo-2003}.

Let $A$ be an $n \times n$ matrix, and $|\det A| = b$. We say $A$ is a dilation matrix if all eigenvalues $\lambda$ of $A$ satisfy $|\lambda| > 1$. If $\lambda_1, \ldots, \lambda_n$ are the eigenvalues of $A$, ordered by their norm from smallest to largest, then define $\lambda_-$ and $\lambda_+$ to satisfy $1 < \lambda_- < |\lambda_1|$ and $|\lambda_n| < \lambda_+$. Given a dilation matrix $A$, we can find a (non-unique) homogeneous quasi-norm, that is, a measurable mapping $\rho_A : \R^n \rightarrow [0, \infty)$ with a doubling constant $c$ satisfying:
        \begin{center}
        \begin{tabular}{lll}
        $\rho_A (x) = 0$
            &   exactly when    &   $x = 0$, \\
        $\rho_A (Ax) = b\rho(x)$
            &   for all         &   $x \in \R^n$, \\
        $\rho_A(x + y) \leq c(\rho_A(x) + \rho_A(y))$  &  for all & $x, y \in \R^n$. \\
        \end{tabular}
        \end{center}
Note that $(\R^n, dx, \rho_A)$ is a space of homogeneous type ($dx$ denotes the Lebesgue measure), and any two quasi-norms associated with $A$ will give the same anisotropic structure. This anisotropic quasi-norm is related to the Euclidean structure by the following lemma of Lemarie-Rieusset \cite{Le-1994}.

\begin{lemma} Suppose $\rho_A$ is a homogeneous quasi-norm associated with dilation $A$. Then there is a constant $c_A$ such that:
    \begin{align}
    \label{EQ1}
    \frac{1}{c_A} \rho_A(x)^{\zeta_-} \leq |x| \leq c_A \rho_A(x)^{\zeta_+}     \ &\textrm{ if } \ \rho_A(x) \geq 1, \\
    \label{EQ2}
    \frac{1}{c_A} \rho_A(x)^{\zeta_+} \leq |x| \leq c_A \rho_A(x)^{\zeta_-}     \ &\textrm{ if } \ \rho_A(x) < 1,
    \end{align}
where $c_A$ depends only on the eccentricities of $A$:  $\displaystyle \zeta_{\pm} = \frac{\ln \lambda_{\pm}}{\ln b}$.
\end{lemma}

In the isotropic setting, the `basic' geometric object is the Euclidean ball $B(x, r)$, centered at $x \in \R^n$ with radius $r$. Conveniently, whenever $r_1 < r_2$, we have $B(x, r_1) \subset B(x, r_2)$. But for a dilation matrix $A$, we do not expect $B(x, r) \subset A(B(x, r))$. Instead, one can construct ellipsoids $\{ B_k \}_{k \in \Z}$, associated with $A$, such that for all $k$, $B_{k + 1} = A(B_k)$, $B_k \subseteq B_{k + 1}$,  and $|B_k| = b^k$. These nested ellipsoids will serve as the basic geometric object in the anisotropic setting. Moreover, we can use the ellipsoids to define the canonical quasinorm associated with $A$ as follows:
    \begin{align}
    \label{CanNorm}
    \rho_A(x) =
            \begin{cases}
            b^j     &\textrm{ if } x \in B_{j + 1} \backslash B_j \\
            0       &\textrm{ if } x = 0.
            \end{cases}
    \end{align}
Once $A$ is fixed, we will drop the subscript and $\rho$ will always denote the canonical norm. If $A^*$ is the adjoint of $A$, then $A^*$ is also a dilation matrix with the same determinant and eccentricities $\zeta_{\pm}$, but with its own nested ellipsoids $\{ B_k^{\ast} \}_{k \in \Z}$ and (canonical) norm $\rho_{\ast}$.

If $k \in \Z$ and $\varphi$ is in the Schwartz class $\mathcal{S}$, with $\int \varphi \ dx \neq 0$, we denote its anisotropic dilation by $\varphi_k (x) = b^{k} \varphi(A^{k} x)$. Then the radial maximal function on $f \in \mathcal{S}'$ is given by
    \[ M_{\varphi}^0 f(x) = \sup_{k \in \Z} |f \ast \varphi_k (x)|. \]
The anisotropic Hardy space $H_A^p$ consists of all tempered distributions $f \in \mathcal{S}'$ so that $M_{\varphi} f \in L^p$. Analogous to the isotropic setting, this definition is independent of the choice of $\varphi$ and is equivalent to the grand maximal function formulation.

In particular, we have the atomic decomposition of $H_A^p$, which greatly simplifies the analysis of Hardy spaces. For a fixed dilation $A$, we say $(p, q, s)$ is an admissible triplet (with respect to $A$) if $p \in (0, 1]$, $1 \leq q \leq \infty, p < q$, and $s \in \N$ satisfying $s \geq \left\lfloor \left( \frac{1}{p} - 1 \right) \frac{1}{\zeta_-} \right\rfloor$. Then a $(p, q, s)$ atom is a function $a(x)$ supported on $x_0 + B_k$ for some $x_0 \in \R^n, k \in \Z$, satisfying
    \begin{center}
    \begin{tabular}{cll}
    (size)       &$\| a \|_q \leq |B_j|^{\frac{1}{q} - \frac{1}{p}}$                & \\
    (vanishing moments)
                &$\ds \int_{\R^n} a(x) x^{\alpha} dx = 0$
                & for \text{all multi-indeces} $|\alpha| \leq s$.
    \end{tabular}
    \end{center}

The standard strategy is to prove a uniform estimate on atoms, and extend it to all $f \in H_A^p$. We will use this strategy for all of our results, possible due to the following atomic characterization, see \cite[Theorem 6.5]{Bo-2003}:
    \begin{theorem} Suppose $p \in (0, 1]$ and $(p, q, s)$ is admissible. Then $f \in H_A^p$ if and only if
        \[ f = \sum_{i} \lambda_i a_i, \]
    for some sequence $( \lambda_i )_i \in \ell^p$ and $( a_i )$ a sequence of $(p, q, s)$ atoms. Moreover,
        \[ \| f \|_{H_A^p} \simeq \inf \{ \| (\lambda_i) \|_{\ell^p}: f = \sum_{i} \lambda_i a_i \}, \]
    where the infimum is taken over all possible atomic decompositions.
    \end{theorem}

\section{Proof of Theorem \ref{Thm1}}
To prepare for the following two lemmas, we recall two basic facts. Define the dilation operator by $D_A (f) (x) = f(Ax)$ commutes with the Fourier transform by the following identity for all $j \in \Z$:
        \begin{align}\label{comm}
        b^j (D_{A^*}^j \mathcal{F} D_A^j f)(\xi) = \hat{f}(\xi).
        \end{align}
Second, the eccentricities of $A^*$ are the same as $A$, that is, \eqref{EQ1} and \eqref{EQ2} hold with the same constants $c_A, \zeta_+, \zeta_-$. Indeed, $A^*$ has the same eigenvalues as $A$.

\begin{lemma} Let $a$ be a $(p, q, s)$ atom supported on $x_0 + B_k$, $k\in \Z$. Suppose $\alpha$ is a multi-index, with $|\alpha| \leq s$. There exists a constant $C = C(s)$ such that
    \begin{align}
    \label{PW-F(a)1}
    |\p^{\alpha} (\mathcal{F}D_A^k a)(\xi)| \leq C b^{-\frac{k}{q}} \| a \|_q \min\{ 1, |\xi|^{s - |\alpha| + 1} \}.
    \end{align}

\end{lemma}

    \begin{proof} Without loss of generality, we can assume $a$ is supported on $B_k$, so  $\supp(D_A^k a) \subset B_0$. Fixing a multi-index $|\alpha| \leq s$, we have
    \begin{align*}
    |\p^{\alpha} (\mathcal{F}D_A^k a)(\xi)|
        &= \bigg|\int_{B_0} (-2\pi i x)^{\alpha} (D_A^k a)(x) e^{-2\pi i \langle x, \xi \rangle} dx \bigg|.
    \end{align*}
    Let $T(x)$ be the degree $s - |\alpha|$ Taylor polynomial of the function $x \mapsto e^{-2\pi i \langle x, \xi \rangle}$ centered at the origin. Using the vanishing moments of an atom, we have
 \begin{align*}
&  |\p^{\alpha} (\mathcal{F}D_A^k a)(\xi)|
        = \bigg|\int_{B_0} (-2\pi i x)^{\alpha} (D_A^k a)(x) e^{-2\pi i \langle x, \xi \rangle} dx \bigg| \\
        &= \bigg|\int_{B_0} (-2\pi i x)^{\alpha} (D_A^k a)(x) \left[ e^{-2\pi i \langle x, \xi \rangle} - T(x) \right] dx \bigg|
  \leq C \int_{B_0} |x^{\alpha}| |a(A^k x)| |x|^{s - |\alpha| + 1} |\xi|^{s - |\alpha| + 1} dx
        \\
&       \leq C |\xi|^{s - |\alpha| + 1} \int_{B_0} |x|^{s + 1} |a(A^k x)| dx
        \leq C |\xi|^{s - |\alpha| + 1} \int_{B_k} |a(y)| \frac{dy}{b^k} \leq C |\xi|^{s - |\alpha| + 1} b^{-k/q}  \| a \|_{q}.
    \end{align*}
The third line is a consequence of Taylor's remainder formula. To obtain the other estimate, we estimate without the Taylor approximation
    \begin{align*}
    |\p^{\alpha} (\mathcal{F} D_A^k a)(\xi)|
        &=      \bigg|\int (-2\pi i x)^{\alpha} (D_A^k a)(x) e^{-2\pi i \langle x, \xi \rangle} dx \bigg|
        \leq C \int_{B_0}  |x|^{|\alpha|} |a(A^k x)| dx\\
        &\leq   C \ b^{-k} \int_{B_k} |a(y )| dy
        \leq C b^{-k/q} \| a \|_q.
    \end{align*}
    \end{proof}

\begin{lemma} Let $a$ be a $(p, q, s)$ atom supported on $x_0 + B_k$ for some $x_0 \in \R^n$ and $k \in \Z$. Then we have the following bound, with $C$ independent of $a$,
    \begin{align}
    \label{PW-F(a)}
    |\hat{a}(\xi)| \leq C \rho_{\ast}(\xi)^{\frac{1}{p} - 1}.
    \end{align}
\end{lemma}

    \begin{proof} Setting $\alpha = 0$, \eqref{PW-F(a)1} reduces to the following estimate
        \begin{align}
        \label{PW-F(a)2}
        |\hat{a}(\xi)| \leq
            \begin{cases}
            C b^{k(1-1/p)} b^{(s + 1) k\zeta_-} \rho_{\ast}(\xi)^{(s + 1) \zeta_-} &\textrm{ for } \rho_{\ast}(\xi) \leq b^{-k}, \\
            Cb^{k(1-1/p)} &\textrm{ for all } \xi.
            \end{cases}
        \end{align}
Indeed, with \eqref{comm} and setting $\alpha = 0$ in \eqref{PW-F(a)1},
    \begin{align*}
    |\hat{a}(\xi)|
        &= |b^{k} (\mathcal{F} D_A^k a) (A^{\ast k} \xi))| \leq   C b^k b^{-\frac{k}{q}} \| a \|_q \min(1,|A^{\ast k} \xi|^{s + 1}) \\
        &\leq   Cb^{k(1-1/p)} \min(1, |A^{\ast k} \xi|^{s + 1}).
    \end{align*}
This immediately yields the second estimate \eqref{PW-F(a)2}. To see the first estimate, we take $\rho_{\ast} (\xi) \leq b^{-k}$, which is equivalent to $A^{*k} \xi \in B_1^*$. Hence, by \eqref{EQ1}, $|(A^*)^k \xi)| \leq c_A b^{k \zeta_-} \rho_{\ast} (\xi)^{\zeta_-}$. Thus,
    \[ |\hat{a}(\xi)| \leq C b^{k(1-1/p)} ( b^{k \zeta_-} \rho_{\ast} (\xi)^{\zeta_-} )^{s + 1}. \]
This shows \eqref{PW-F(a)2}, which we will use to prove \eqref{PW-F(a)}.

If $\rho_{\ast}(\xi) \leq b^{-k}$, then
\begin{align*}
        |\hat{a}(\xi)|
            &\leq C b^{k((1-1/p)+(s + 1) \zeta_-)} \rho_{\ast}(\xi)^{(s + 1) \zeta_-}  \\
            &\leq C \rho_*(\xi)^{-(1-1/p)-(s + 1) \zeta_-} \rho_{\ast}(\xi)^{(s + 1) \zeta_-} = C \rho_{\ast} (\xi)^{\frac{1}{p} - 1}.
        \end{align*}
In the second inequality we used the fact that $1 - \frac{1}{p} + (s + 1) \zeta_- \geq 0$, since $(p, q, s)$ is be admissible. If $\rho_{\ast}(\xi) > b^{-k}$, then by \eqref{PW-F(a)2}, we have
        \begin{align*}
        |\hat{a}(\xi)|  &\leq C b^{-k(1/p-1)} \leq  C \rho_{\ast}(\xi)^{\frac{1}{p} - 1}.
        \end{align*}
where the last inequality holds since $1/p-1\ge 0$. This completes the proof of the lemma.
    \end{proof}

We are now ready to prove Theorem \ref{Thm1} by extending \eqref{PW-F(a)} to every $f \in H_A^p$.

\begin{proof}[Proof of Theorem \ref{Thm1}] Let $f \in H_A^p$. By the atomic decomposition of $H_A^p$, we can find coefficients $(\lambda_i)$ and atoms $(a_i)$ such that $f = \sum \lambda_i a_i$ (in $H_A^p$-norm) and $2\| f \|_{H_A^p}  \geq \| (\lambda_i) \|_{\ell^p}$. This sum converges in $H_A^p$-norm, which implies convergence in $\mathcal S'$. So by taking the Fourier transform on $f$, we have $\hat{f} = \sum_{i} \lambda_i \hat{a_i}$, converging in $\mathcal S'$. By \eqref{PW-F(a)} and the fact that $(\lambda_i) \in \ell^1$,
    \begin{align*}
    \sum_{i = 1}^{\infty} |\lambda_i| |\hat{a}_i (\xi)| \leq C \sum_{i = 1}^{\infty} |\lambda_i| \ \rho_{\ast}(\xi)^{\frac{1}{p} - 1}
    \leq 2C \rho_{\ast} (\xi)^{\frac{1}{p} - 1} \| f \|_{H^p}  < \infty.
    \end{align*}
Therefore, the sum $\hat{f}(\xi) = \sum_{i} \lambda_i \hat{a}_i (\xi)$ converges absolutely on $\R^n$. Furthermore, on each compact set $K$, $\rho_{\ast}(\xi)$ is bounded by a constant $C'$ independent of $a$, so the absolute convergence above is also uniform on each compact set $K$. With $\hat{a_i}$ infinitely differentiable (hence continuous) for all $i$, we conclude $\hat{f}(\xi)$ is continuous on all compact sets $K$, and hence on $\R^n$.
    \end{proof}

\section{Applications of Theorem \ref{Thm1}}

We now consider consequences of Theorem \ref{Thm1}. The first corollary refines the order of 0 at the origin, and the second gives necessary conditions on a multiplier $m$ on $H_A^p$. The third corollary is the Hardy-Littlewood inequality on Hardy spaces, which will be strengthened by a rearrangement argument.

\begin{corollary}\label{Cor-Local} Let $f \in H_A^p(\R^n)$, $0<p \le 1$. Then,
    \begin{align}
    \label{PW-F(f)-origin}
    \lim_{\xi \rightarrow 0} \frac{\hat{f}(\xi)}{\rho_{\ast}(\xi)^{\frac{1}{p} - 1}} = 0.
    \end{align}
\end{corollary}

    \begin{proof} We start by verifing this on an atom $a$, with support $B_k$. By \eqref{PW-F(a)2}, if $\rho_{\ast}(\xi) \leq b^{-k}$, we have
\[
 |\hat{a}(\xi)| \leq C b^{k(1-1/p)} b^{(s + 1) k\zeta_-} \rho_{\ast}(\xi)^{(s + 1) \zeta_-}.
\]
Since $s \geq \lfloor (1/p - 1) \zeta_- \rfloor$, this implies $(s + 1) \zeta_{-} > \frac{1}{p} - 1$. Therefore, we obtain \eqref{PW-F(f)-origin} for atoms;
     \[ \lim_{\xi \rightarrow 0} \frac{\hat{a}(\xi)}{\rho_{\ast}(\xi)^{\frac{1}{p} - 1}} = 0. \]
Now if $f \in H_A^p$, we can decompose $f = \sum_{i} \lambda_i a_i$, for $(\lambda_i) \in \ell^p$ and $(p, q, s)$-atoms $a_i$. Thus,
        \[ \frac{|\hat{f}(\xi)|}{\rho_{\ast}(\xi)^{\frac{1}{p} - 1}} \leq \sum_{i = 1}^{\infty} \frac{|\hat{a}_i (\xi)|}{\rho_{\ast}(\xi)^{\frac{1}{p} - 1}} |\lambda_i|. \]
By \eqref{PW-F(a)} and the fact that $(\lambda_i) \in \ell^1$, we can apply the Dominated Convergence Theorem to the above sum (treated as an integral). Since each term in the sum goes to $0$ as $\xi \to 0$ we obtain \eqref{PW-F(f)-origin}.
    \end{proof}

Corollary \ref{Cor-Mult}, which is a generalization of \cite[Theorem III.7.31]{GCF-WN}, gives a necessary condition for multipliers on anisotropic Hardy spaces $H^p_A$.

\begin{corollary}\label{Cor-Mult} Suppose $m$ is a multiplier on $H_A^p$, $0<p \le 1$. That is, the following operator is bounded:
    \[ T_m : H_A^p \rightarrow H_A^p, \qquad T_m (f) = (m \hat{f})^{\vee}, \]
with $M > 0$ as the operator norm of $T_m$. Then, $m$ is continuous on $\R^n \bs \{ 0 \}$ and uniformly bounded with $\| m \|_{\infty} \leq CM$.
\end{corollary}

\begin{proof} Fix $0<p\le 1$. For $k \in \Z$, we denote $f_k (x) = b^{k/p} f(A^k x)$. Then, this dilation is invariant under $H^p_A$ (and $L^p$) norm: $\| f_k \|_{H^p_A} = \| f \|_{H^p_A}$. Under the Fourier transform, we have
        \[ \hat{f_k} (\xi) = b^{k(\frac{1}{p} - 1)} \hat{f}((A^*)^{-k} \xi). \]
    Then by \eqref{Aniso-F(f)}, the following estimate holds for all $k \in \Z, \xi \in \R^n$,
        \[ |m(\xi) \hat{f}((A^*)^{-k} \xi)| \leq C M \| f \|_{H^p} \ \rho_{\ast} (\xi)^{\frac{1}{p} - 1} b^{k(1 - \frac{1}{p})} . \]
    If $\xi \in B_{k + 1}^* \bs B_k^*$, then $(A^*)^{-k} \xi \in B_1^* \bs B_0^*$, and we have
        \[ |m(\xi) \hat{f} ((A^*)^{-k} \xi)| \leq C M \| f \|_{H^p}. \]
    This estimate will force $m$ to be bounded if we there exists $f \in H_A^p$ such that $\hat{f}$ does not vanish on the unit annulus $B_1^* \bs B_0^*$. Take $g \in C_c^{\infty}$, supported on $B_{2}^* \bs B_{-1}^*$ such that $g$ is identically 1 on $B_1^* \bs B_0^*$. Setting $\hat{f} = g$, $f$ is immediately in the Schwartz class $\mathcal S$, with vanishing moments of all order. In particular, $f$ is a molecule for $H_A^p$ (see Remark in \cite[Section 9]{Bo-2003}), hence $f\in H_A^p$. This shows that $||m||_\infty \le CM$. Moreover, by Theorem \ref{Thm1} the function $\xi \mapsto m(\xi) \hat{f} ((A^*)^{-k} \xi)$ is continuous for each $k\in \Z$. Thus, $m$ is continuous on $\R^n \bs \{ 0 \}$.
\end{proof}

\begin{corollary}\label{Cor-Global} If $f \in H_A^p(\R^n)$, $0 < p \leq 1$, then
    \begin{align}
    \label{HL-F(f)}
    \int_{\R^n} |\hat{f}(\xi)|^p \rho_{\ast} (\xi)^{p - 2} d\xi \leq C \| f \|_{H^p_A}^p.
    \end{align}
\end{corollary}

    \begin{proof} Suppose a $(p, 2, s)$ atom $a$ is supported on $x_0+B_k$.
We claim that
    \begin{align}
    \label{HL-F(a)}
    \int_{\R^n} |\hat{a}(\xi)|^p \rho_{\ast}(\xi)^{p - 2} d\xi \leq C.
    \end{align}
Indeed, by \eqref{PW-F(a)2} we can estimate the integral on $B_{-k}^{\ast}$
        \begin{align*}
        \int_{B_{-k}^{\ast}} |\hat{a}(\xi)|^p \rho_{\ast} (\xi)^{p - 2} d\xi
            &\leq C^p  b^{k(p-1)} b^{p(s + 1) k\zeta_-}   \int_{B_{-k}^{\ast}} \rho_{\ast} (\xi)^{p - 2 + p(s + 1) \zeta_-} d\xi \leq C^p.
        \end{align*}
    For the integral outside of $B_{-k}^{\ast}$, we use H\"{o}lder's inequality
        \begin{align*}
        \int_{(B_{-k}^\ast)^c} |\hat{a}(\xi)|^p \rho_{\ast} (\xi)^{p - 2} d\xi
            &\leq C \left( \int_{(B_{-k}^\ast)^c} |\hat{a}(\xi)|^2 d\xi \right)^{\frac{p}{2}} \left( \int_{(B_{-k}^\ast)^c} \rho_{\ast} (\xi)^{-2} d\xi \right)^{\frac{2 - p}{2}} \\
            &\leq C ||a||_2^{p} b^{-k(\frac{p}{2} - 1)}\leq C.
        \end{align*}
    Combining these two estimates, we obtain \eqref{HL-F(a)}. Now let $f \in H_A^p$ have an atomic decomposition $f = \sum_i \lambda_i a_i$ with $ \| (\lambda_i) \|_{\ell^p} \le 2\| f \|_{H_A^p}  $. Since $p \in (0, 1]$, we have
        \begin{align*}
        \int_{\R^n} |\hat{f}(\xi)|^p \rho_{\ast}(\xi)^{p - 2} d\xi
            &\leq \sum_i |\lambda_i|^p \int_{\R^n} |\hat{a}_i (\xi)|^p \rho_{\ast}(\xi)^{p - 2} \ d\xi \leq C \sum_i |\lambda_i|^p
            \leq C \| f \|_{H^p_A}^p.
        \end{align*}
This shows \eqref{HL-F(f)}.
    \end{proof}

The following result improves \eqref{HL-F(f)} by extending \cite[Lemma 3.1]{GCK-2001} to the anisotropic setting. We denote $S_0(\R^n)$ as the collection of all measurable functions $f$, finite almost everywhere, whose distributional functions satisfy \begin{equation}\label{df}
d_f (t) = |\{ x \in \R^n : |f(x)| > t \}| < \infty \qquad\text{for all }t > 0.
\end{equation}
For $f \in S_0(\R^n)$, its rearrangement function is defined by
    \[ f^{\star} (t) = \inf \{ s > 0 : d_f (s) \leq t \}. \]
We recall the following facts regarding the rearrangement function. If $f \leq g$ on $\R^n$, then $f^{\star} (t) \leq g^{\star} (t)$ for all $t$. For all $\lambda > 0$,
\begin{equation}\label{Rear-com}
(|f|^{\lambda})^{\star} (t) = f^{\star}(t)^{\lambda}.
\end{equation}
These follow immediately from the definition. Lastly,
    \begin{align}\label{Rear-Subl}
      \int_0^t \left( \sum_j f_j \right)^{\star} (u) du \leq \sum_j \int_0^t f_j^{\star} (u) du,
    \end{align}
    for all $t > 0$, provided the right-hand side is finite; see \cite[Chapter 2, \S 3]{BS-IoO}.

\begin{theorem} \label{Lem-Rearr} Let $\epsilon > 0, 0<p < 1$ and define $\lambda = \frac{1}{p} - 1 + \epsilon$. Then, there exists $C$ such that for all $f \in H_A^p(\R^n)$,
    \begin{align}
    \label{A3.2}
    \left( \int_0^{\infty} t^{\epsilon p - 1} F_{\epsilon}^{\star} (t)^p dt \right)^{1/p} \leq C \| f \|_{H_A^p},
    \end{align}
with $F_{\epsilon} (\xi) = \rho_{\ast} (\xi)^{-\lambda} |\hat{f}(\xi)|$.
\end{theorem}

To see why Theorem \ref{Lem-Rearr} strengthens \eqref{HL-F(f)}, we observe that if $g(\xi) = 1/\rho_{\ast}(\xi)$, a simple computation shows
    \begin{align}\label{Rearr-rho}
    g^{\star} (t) \simeq 1/t.
    \end{align}
If $f, g \in S_0(\R^n)$,
    \begin{align*}
    \int_{\R^n} |f(\xi) g(\xi)| dx \leq \int_0^{\infty} f^{\star} (t) g^{\star} (t) dt.
    \end{align*}
Together, these two facts can be used to show the left-hand side of \eqref{A3.2} majorizes the left-hand side of \eqref{HL-F(f)}.

\begin{proof}[Proof of Theorem \ref{Lem-Rearr}] We will prove the following estimate for all $f \in H_A^p$:
        \begin{align}
        \label{A3.4}
        \left( \int_0^{\infty} t^{\epsilon p - 2} \left[ \int_0^t F_{\epsilon}^{\star} (u)^p du \right] dt \right)^{1/p} \leq C \| f \|_{H_A^p},
        \end{align}
which implies \eqref{A3.2}. Indeed, the rearrangement function is always decreasing for $0 < t < \infty$. Thus, $F_{\epsilon}^{\star}(t)^p \leq \frac{1}{t} \int_0^t F_{\epsilon}^{\star} (u)^p du$. Then,
    \begin{align*}
    \int_0^{\infty} t^{\epsilon p - 1} F_{\epsilon}^{\star} (t)^p dt
        &\leq \int_0^{\infty} t^{\epsilon p - 1} \left( \frac{1}{t} \int_0^t F_{\epsilon}^{\star} (u)^p du \right) dt = \int_0^{\infty} t^{\epsilon p - 2} \left( \int_0^t F_{\epsilon}^{\star} (u)^p du \right) dt.
    \end{align*}
We first prove \eqref{A3.4} for unit atoms. Using a dilation argument, we extend it to all atoms, and to any $f \in H_A^p$ using the atomic decomposition.

Let $f$ be a unit $(p, 2, s)$ atom, that is, an atom supported on $x_0 + B_0$. Without loss of generality, we set $x_0 = 0$.
On unit atoms, the estimates \eqref{PW-F(a)} and \eqref{PW-F(a)2} reduce to
        \begin{align*}
        \| \hat{f} \|_{\infty} \leq
            \begin{cases}
            \rho_{\ast}(\xi)^{(s + 1) \zeta_-}
                &\textrm{ for } \xi \in B_0^* \\
            \rho_{\ast}(\xi)^{\frac{1}{p} - 1}
                &\textrm{ for all } \xi.
            \end{cases}
        \end{align*}
This implies
    \begin{align*}
    F_{\epsilon} (\xi)
        &\leq
            \begin{cases}
            \rho_{\ast}(\xi)^{\zeta_- (s + 1)  - \lambda}
                &\textrm{ for } \xi \in B_0^*, \\
            \rho_{\ast}(\xi)^{\frac{1}{p} - 1 - \lambda}
                &\textrm{ for all } \xi,
            \end{cases}
    \end{align*}
where the first estimate has a positive power, and the second has a negative power. These give  $\| F_{\epsilon} \|_{\infty} \leq C$, and $F_{\epsilon} (\xi) \leq C \rho_{\ast} (\xi)^{-\lambda}$, which by the properties of the rearrangement function and \eqref{Rearr-rho}, imply
    \[ F_{\epsilon}^{\star} (t) \leq C \min\{ 1, t^{-\lambda} \}. \]
With these estimates,
    \begin{align*}
    \int_0^{\infty} t^{\epsilon p - 2} \left( \int_0^t F_{\epsilon}^{\star} (u)^p du \right) dt = \int_0^{1} + \int_1^{\infty} t^{\epsilon p - 2} \left( \int_0^t F_{\epsilon}^{\star} (u)^p du \right) dt = I_1 + I_2.
    \end{align*}
By the fact that $F_{\epsilon}^{\star} (t) \leq C$, we have $I_1 \leq C$. To estimate $I_2$,
    \begin{align*}
    I_2
        &\leq  \int_1^{\infty} t^{\epsilon p - 2} \left( \int_0^t F_{\epsilon}^{\star} (u)^p du \right) dt \leq  \int_1^{\infty} t^{\epsilon p - 2} \left( \int_0^t u^{-\lambda p} du \right) dt \\
        &\simeq \int_1^{\infty} t^{\epsilon p - 2} t^{1 - \lambda p} dt = \int_1^{\infty} t^{p - 2} dt \leq C.
    \end{align*}
Since $\| f \|_{H_A^p} \leq C$ for all atoms, we have \eqref{A3.4} for unit atoms.

We now extend it to all atoms using a dilation argument. Let $f$ be a general $(p, 2, s)$ atom supported on $B_{k}$. Then the dilated atom $f_k (x) = b^{k/p} f(A^k x)$ is an atom with the same $H_A^p$-norm, but supported on $B_0$, that is, $f_k$ is a unit atom. Denoting $G_{\epsilon} (\xi) = \rho_{\ast}(\xi)^{-\lambda} \ |\widehat{f_k} (\xi)|$, we have just shown that
    \[ \int_0^{\infty} t^{\epsilon p - 2} \left( \int_0^t G_{\epsilon}^{\star} (u)^p du \right) dt \leq C. \]
The fact that \eqref{A3.4} holds for all atoms follows if we can show that the above quantity is the same if we replace $G_{\epsilon}$ by $F_{\epsilon} (\xi) = \rho_{\ast} (\xi)^{-\lambda} |\hat{f}(\xi)|$.

As before, we denote $D_{A^{\ast}} g(x) = g(A^{\ast} x)$. Then
    \[ G_{\epsilon} (\xi) = b^{-\epsilon k} (D_{A^*}^k F_{\epsilon})(\xi). \]
The distribution function is affected as follows.
    \begin{align*}
    d_{G_{\epsilon}} (s)
        &= |\{ \xi : G_{\epsilon} (\xi) > s \}| = |\{ \xi : (D_{A^*}^{-k} F_{\epsilon})(\xi) > sb^{\epsilon k} \}| \\
        &= |\{ \xi : F_{\epsilon} ((A^{\ast})^{-k} \xi) > sb^{\epsilon k} \}| = b^k |\{ u : F_{\epsilon} (u) > sb^{\epsilon k} \}| = b^k d_{F_{\epsilon}} (sb^{\epsilon k}).
        \end{align*}
This affects the rearrangement function as follows:
\begin{equation}\label{G-ep}
\begin{aligned}
    G_{\epsilon}^{\star} (t)
        &= \inf \{ s : d_{G_{\epsilon}} (s) \leq t \} = \inf \{ s : d_{F_{\epsilon}} (sb^{\epsilon k}) \leq b^{-k} t \} \\
        &= b^{-\epsilon k} \inf \{ r : d_{F_{\epsilon}} (r) \leq b^{-k} t \} = b^{-\epsilon k} F_{\epsilon}^{\star} (b^{-k} t).
\end{aligned}
\end{equation}
By two changes of variables and \eqref{G-ep}, we have
    \begin{align*}
    \int_0^{\infty} t^{\epsilon p - 2} \left( \int_0^t F_{\epsilon}^{\star} (u)^p du \right) dt
        &= \int_0^{\infty} t^{\epsilon p - 2} \left( \int_0^t b^{p \epsilon k} G_{\epsilon}^{\star} (b^{k} u)^p du \right) dt \\
        &= \int_0^{\infty} s^{\epsilon p - 2} \left( \int_0^s G_{\epsilon}^{\star} (r)^p dr \right) ds \leq C.
    \end{align*}
This extends \eqref{A3.4} to all atoms, and we now extend it to all $f \in H_A^p$.
 If $f \in H_A^p$, then we have the atomic decomposition
    \[ f = \sum_{j} \lambda_j a_j, \]
with $(p, 2, s)$ atoms $a_j$ and $(\lambda_j) \in \ell^p$. Taking the Fourier transform, we have the following sum in the distributional and pointwise sense:
    \[ \hat{f}(\xi) = \sum_j \lambda_j \widehat{a_j}(\xi). \]
With $F_{\epsilon} (\xi) = |\xi|^{-\lambda} |\hat{f}(\xi)|$ and $p \in (0, 1)$,
    \[ F_{\epsilon} (\xi)^p = \left( |\xi|^{-\lambda} | \sum_j \lambda_j \widehat{a_j} (\xi)| \right)^p \leq \sum_j |\lambda_j|^p \cdot \left( |\xi|^{-\lambda} |\widehat{a_j}(\xi)| \right)^p = \sum_j |\lambda_j|^p A_j (\xi)^p , \]
where $A_j (\xi) = |\xi|^{-\lambda} |\widehat{a_j} (\xi)|$. Recall that the rearrangement operation is order-preserving ($f \leq g \Rightarrow f^{\star} \leq g^{\star})$. By \eqref{Rear-com} and \eqref{Rear-Subl}, we have
    \begin{align*}
    \int_0^t F_{\epsilon}^{\star} (u)^p du
        &\leq \int_0^t \left( \sum_j |\lambda_j|^p A_j^p (\cdot) \right)^{\star} (u) du \leq \sum_j |\lambda_j|^p \int_0^t A_j^{\star} (u)^p du.
    \end{align*}
Therefore,
    \begin{align*}
    \int_0^{\infty} t^{\epsilon p - 2} \left[ \int_0^t F_{\epsilon}^{\star} (u)^p du \right] dt
        &\leq \int_0^{\infty} t^{\epsilon p - 2} \left[ \sum_j |\lambda_j|^p \int_0^t A_j^{\star} (u)^p du \right] dt \\
        &= \sum_j |\lambda_j|^p \int_0^{\infty} t^{\epsilon p - 2} \left( \int_0^t A_j^{\star} (u)^p du \right) dt \leq C \sum_j |\lambda_j|^p,
    \end{align*}
where the last inequality comes from $\eqref{A3.4}$ for all atoms. Taking the infimum over all possible atomic decompositions, we obtain \eqref{A3.4} for all $f \in H_A^p$.
\end{proof}

\end{document}